\documentclass{amsart}

\usepackage{amsmath,amssymb,amsthm}

\usepackage{pinlabel}

\newtheorem{prop}{Proposition}
\newtheorem{thm}[prop]{Theorem}
\newtheorem{lem}[prop]{Lemma}

\theoremstyle{definition}

\newtheorem{rem}[prop]{Remark}

\def\co{\colon\thinspace}

\newcommand{\alphast}{\alpha_{\mathrm{st}}}

\newcommand{\rmd}{\mathrm{d}}

\newcommand{\ogamma}{\overline{\gamma}}

\newcommand{\R}{\mathbb{R}}

\newcommand{\xist}{\xi_{\mathrm{st}}}

\DeclareMathOperator{\curl}{curl}

\begin{document}

\author[T. Becker]{Tilman Becker}
\author[H.~Geiges]{Hansj\"org Geiges}
\address{Mathematisches Institut, Universit\"at zu K\"oln,
Weyertal 86--90, 50931 K\"oln, Germany}
\email{tibecker@math.uni-koeln.de, geiges@math.uni-koeln.de}

\thanks{H.~G. is partially supported by the SFB/TRR 191
`Symplectic Structures in Geometry, Algebra and Dynamics',
funded by the DFG}

\title[Line fibrations and contact structures]{The contact structure
induced by a line fibration of $\R^3$ is standard}

\date{}

\begin{abstract}
Building on the work of and answering a question by Michael Harrison,
we show that any contact structure on $\R^3$ induced by a line
fibration of $\R^3$ is diffeomorphic to the standard contact structure.
\end{abstract}

\keywords{}

\subjclass[2010]{53C12, 53D35}

\maketitle


\section{Introduction}
A  fibration of $\R^3$ by oriented lines (or \emph{line fibration},
for short) is described by a smooth vector field $V$ on $\R^3$ whose
integral curves are straight lines. We may write
\[ V=V_1\partial_x+V_2\partial_y+V_3\partial_z\]
and, assuming $V$ to be of unit length, regard
\[ V=(V_1,V_2,V_3)\co\R^3\longrightarrow S^2\subset\R^3\]
as a map into the unit $2$-sphere.
We write
\[  \ell_p=\bigl\{p+tV(p)\co t\in\R\bigr\}\]
for the line through $p$ defined by~$V$, and $\ell_p^{\perp}\subset\R^3$
for the affine plane through~$p$, orthogonal to~$\ell_p$.
We shall write $\{V\}$ for the line fibration defined by the vector field~$V$.
Any line fibration $\{V\}$ defines a tangent $2$-plane distribution $\xi$
on $\R^3$ by
\[ \xi_p=\langle V(p)\rangle^{\perp}\subset T_p\R^3.\]

Line fibrations of $\R^3$ were studied by M.~Salvai~\cite{salv09}. Building
on Salvai's work, M.~Harrison~\cite{harr19} showed that the $2$-plane
distribution $\xi$ defined by any non-degene\-rate line fibration $\{V\}$
is a tight contact structure, and hence (by Eliashberg's
classification) diffeomorphic (perhaps orientation-reversingly) to the
standard contact structure
\[ \xist=\ker(\rmd z+ x\,\rmd y)\]
on $\R^3$; see Section~\ref{section:nondeg} for the definition
of non-degeneracy. Harrison also gave a sufficient criterion for tightness
under the \emph{a priori} assumption that the induced
$2$-plane distribution $\xi$ be a contact structure.
He then posed the question whether there exists a line
fibration of $\R^3$ that induces an overtwisted contact structure, and
conjectured the answer to be no.

In the present note we confirm this conjecture. In fact, we show directly
(in the case not covered by Harrison's work)
that any contact structure arising from a line fibration is diffeomorphic
to the standard one.

\begin{thm}
\label{thm:main}
If a line fibration $\{V\}$ of $\R^3$ defines a contact structure~$\xi$,
then $\xi$ is diffeomorphic to~$\xist$.
\end{thm}
\section{Non-degenerate implies skew}
\label{section:nondeg}
A line fibration $\{V\}$ is called \emph{skew} if it does
not contain any distinct parallel lines. It is called \emph{non-degenerate}
if the differential $\rmd V$ vanishes only in the direction
of~$V$, in other words, if $\rmd_pV|_{\xi_p}$ is of rank~$2$
everywhere.

The plane field $\xi$ defined by the line fibration $\{V\}$ is given as
the kernel of the $1$-form
\[ \alpha=V_1\,\rmd x_1+V_2\,\rmd x_2+V_3\,\rmd x_3.\]
A straightforward computation shows that the contact condition
$\alpha\wedge\rmd\alpha\neq 0$ is equivalent to
\[ \langle V,\curl V\rangle\neq 0.\]

In \cite[Theorem~1]{harr19} it is shown that this condition is
satisfied whenever the line fibration is non-degenerate.

The following lemma, which says that non-degenerate line
fibrations are skew, is essentially \cite[Lemma~6]{salv09},
but we give a more specific statement and a more direct proof.

\begin{lem}
\label{lem:nondegenerate-skew}
If $\rmd_{p_0}V$ is of rank~$2$
at some point $p_0\in\R^3$, then
the line fibration $\{V\}$ does not contain any line $\ell\neq\ell_{p_0}$
parallel to $\ell_{p_0}$.
\end{lem}

\begin{proof}
By assumption, the restriction of the differential $\rmd_{p_0}V$
to the plane $\xi_{p_0}\subset T_{p_0}\R^3$ defines an isomorphism
$\xi_{p_0}\rightarrow T_{V(p_0)}S^2$.
It follows that we can find a small disc
$D^2_{\varepsilon}\subset\ell_{p_0}^{\perp}$
about $p_0$ that is mapped diffeomorphically by $V$ onto a small
neighbourhood of $V(p_0)$ in~$S^2$. In particular,
the map $V$ never attains the value $\pm V(p_0)$
on~$D^2_{\varepsilon}\setminus\{p_0\}$.

This means that the
vector $W(p)$ obtained by projecting $V(p)$ orthogonally to
$\ell_{p_0}^{\perp}$ is non-zero for
$p\in D^2_{\varepsilon}\setminus\{p_0\}$. We also know that $W(p)$ cannot
point along the line through $p_0$ and $p$, or else $\ell_p$ would
intersect~$\ell_{p_0}$. This implies that $W(p)$ completes exactly one
positive turn as $p$ ranges along $\partial D^2_{\varepsilon}$, see
Figure~\ref{figure:parallel-lines}, which shows the situation
in the affine plane~$\ell_{p_0}^{\perp}$.

\begin{figure}[h]
\labellist
\small\hair 2pt
\pinlabel $p_0$ [bl] at 267 146
\pinlabel $p_1$ [tl] at 491 341
\pinlabel $p$ [br] at 139 147
\pinlabel $W(p)$ [tl] at 180 168
\endlabellist
\centering
\includegraphics[scale=.6]{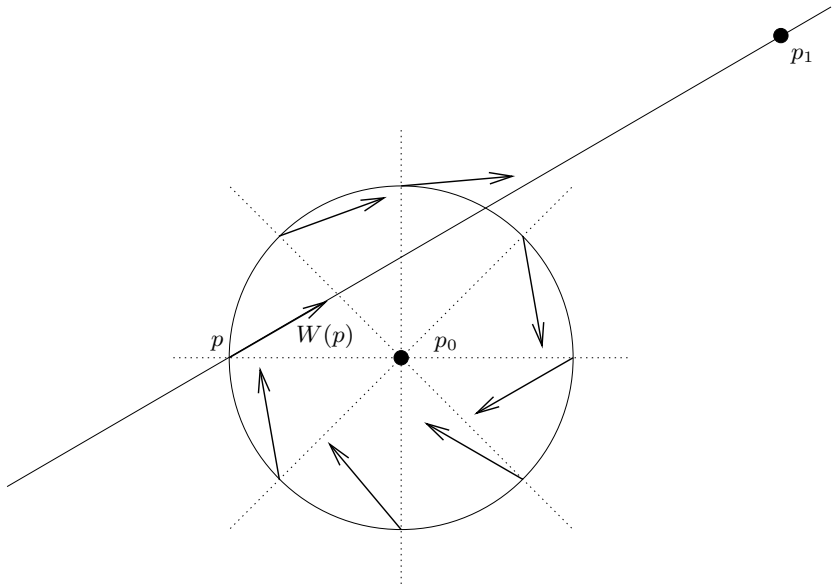}
  \caption{Non-degenerate implies skew.}
  \label{figure:parallel-lines}
\end{figure}

Now suppose there was a line $\ell_{p_1}\neq\ell_{p_0}$ parallel to
$\ell_{p_0}$ in the line fibration $\{V\}$. We may take $p_1$ to be the
intersection point of this line with~$\ell_{p_0}^{\perp}$. Notice that $p_1$
has to lie outside the $2$-disc $D^2_{\varepsilon}$.
As $p$ ranges along $\partial D^2_{\varepsilon}$, the vector
$p-p_1$ never completes a full turn.

It would follow that there have to be at least two points
$p\in\partial D^2_{\varepsilon}$ where $W(p)$ is a scalar multiple
of $p-p_1$, which would imply that $\ell_p$ intersects~$\ell_{p_1}$,
contradicting the assumption that $V$ defines a line fibration.
\end{proof}
\section{Proof of Theorem~\ref{thm:main}}
We begin with a simple lemma.

\begin{lem}
\label{lem:parallel}
If the line fibration $\{V\}$ defines a contact structure, and for every line
$\ell\in\{V\}$ there exists a parallel line $\ell'\neq\ell$ in~$\{V\}$,
then $\rmd V$ has rank~$1$ at every point of~$\R^3$.
\end{lem}

\begin{proof}
If $\alpha=V_1\,\rmd x_1+V_2\,\rmd x_2+V_3\,\rmd x_3$ is a contact
form, then $\rmd V$ is nowhere trivial, since
$\rmd\alpha=\sum_i\rmd V_i\wedge\rmd x_i$. On the other hand, the
assumption on parallel lines implies with Lemma~\ref{lem:nondegenerate-skew}
that the rank of $\rmd V$ can nowhere be equal to~$2$.
\end{proof}

\begin{proof}[Proof of Theorem~\ref{thm:main}]
As shown by Harrison~\cite[Theorem~2]{harr19}, if there is a line in
$\{V\}$ not parallel to any other line in~$\{V\}$, then
$\xi$ is tight, and hence diffeomorphic to $\xist$
by Eliashberg's classification~\cite{elia92},
cf.~\cite[Theorem~4.10.1]{geig08}.

It remains to consider the case when
for every line $\ell\in\{V\}$ there exists a parallel
line $\ell'\neq\ell$ in~$\{V\}$.
We are then in the situation of Lemma~\ref{lem:parallel},
so $\rmd V$ has rank~$1$ at every point of~$\R^3$.

Thus, $\ker\rmd V\cap\xi$ defines a line field on $\R^3$. This
line field is necessarily trivial, and we choose
a vector field $X$ of constant length~$1$ spanning it.
Any flow line of $X$ stays within a compact region in finite time.
It follows that $X$ has a global flow, defined for all times.

Fix a point $p_0\in\R^3$ and consider the maximal flow line
$\gamma_0\co\R\rightarrow\R^3$ of $X$ through $p_0=\gamma_0(0)$.
Then $V(\gamma_0(t))=V(p_0)=:V_0$ for all $t\in\R$, since
\[ \rmd_{\gamma_0(t)} V\bigl((\gamma_0'(t)\bigr)=
\rmd_{\gamma_0(t)} V\bigl(X(\gamma_0(t))\bigr)=0.\]

Write $\pi\co\R^3\rightarrow\ell_{p_0}^{\perp}$ for the orthogonal projection
of $\R^3$ along $V_0$ onto the affine plane $\ell_{p_0}^{\perp}$.
The projected curve $\ogamma_0=\pi\circ\gamma_0$ is regular, since
$X$ is orthogonal to~$V_0$ along~$\gamma_0$. Notice that $V$ is also
constant equal to $V_0$ along~$\ogamma_0$.

We claim that $\ogamma_0$ has to be a straight line. Otherwise,
$\ogamma_0$ would have non-vanishing curvature at some
point $\ogamma_0(t_0)$. Then, for $\delta>0$ sufficiently small,
the arc $A=\ogamma_0(I_0)$, where $I_0=[t_0-\delta,t_0+\delta]$,
would lie to one side of the secant $S$ joining $\ogamma_0(t_0-\delta)$
with~$\ogamma_0(t_0+\delta)$. It follows that $V(q)=V_0$ also for
all $q\in S$, or else $\ell_q$ would intersect one of the $\ell_p$
with $p\in A$. So the lines in $\{V\}$ through
$A$ and $S$ form a straight cylinder.
But then in fact $V\equiv V_0$ on the whole disc-like region
in $\ell_{p_0}^{\perp}$ bounded by $A$ and~$S$, which would mean
that $\rmd V\equiv 0$ there, contradicting the fact
that $\rmd V$ has rank~$1$ everywhere.

Hence, the $V$-lines through $\ogamma_0$ form an affine $2$-plane.
Choose coordinates on $\R^3$ such that this plane
coincides with the $xy$-plane, $\ogamma_0$ with the $y$-axis,
and $V_0=\partial_x$ (hence $V=\partial_x$ along the $xy$-plane).
Then all lines in $\{V\}$ must lie in affine planes parallel
to the $xy$-plane, or else there would be an intersection of lines,
and the lines in each of these horizontal planes must be parallel.
It follows that
\[ V(x,y,z)=\bigl(\cos \theta(z),-\sin \theta(z), 0\bigr)\]
for some smooth function $\theta\co\R\rightarrow\R$,
with $\theta(0)=0$ and, since $\rmd V$ has rank~$1$, satisfying
$\theta'(z)\neq 0$ for all $z\in\R$.

A diffeomorphism of $\R^3$ pulling back
\[ \alpha=\cos\theta(z)\,\rmd x-\sin\theta(z)\,\rmd y\]
to the standard contact form $\alphast=dz+x\,\rmd y$ is given by
\[ (x,y,z)\longmapsto
\Bigl(z\cos\theta(y)+\frac{x}{\theta'(y)}\sin\theta(y),
-z\sin\theta(y)+\frac{x}{\theta'(y)}\cos\theta(y),y\Bigr). \]
This completes the proof of Theorem~\ref{thm:main}.
\end{proof}

\begin{rem}
Observe that this argument yields another sufficient condition
for a line fibration of $\R^3$ to induce a contact structure: if $\rmd V$
has constant rank~$1$, then $\{V\}$ defines the standard contact structure
$\xist$ on~$\R^3$.
\end{rem}

\begin{thebibliography}{10}
%
\bibitem{elia92}
\textsc{Ya. Eliashberg},
Contact $3$-manifolds twenty years since J.~Martinet's work,
\textit{Ann. Inst. Fourier (Grenoble)}
\textbf{42} (1992), 165--192.
%
\bibitem{geig08}
\textsc{H. Geiges},
\textit{An Introduction to Contact Topology},
Cambridge Stud. Adv. Math. \textbf{109},
Cambridge University Press, Cambridge (2008).
%
\bibitem{harr19}
\textsc{M. Harrison},
Contact structures induced by skew fibrations of $\R^3$,
\texttt{arXiv:1904.00405}.
%
\bibitem{salv09}
\textsc{M. Salvai},
Global smooth fibrations of $\R^3$ by oriented lines,
\textit{Bull. London Math. Soc.}
\textbf{41} (2009), 155--163.
%
\end{thebibliography}
\end{document}